\documentclass[12pt,psfig,reqno]{amsart}
\usepackage{hyperref}
\usepackage{amssymb}
\usepackage{amsmath}
\usepackage{amsthm}
\usepackage{enumerate}
\usepackage{graphicx}
\usepackage{color}
\usepackage{cite}
\makeatletter
\@namedef{subjclassname@2020}{\textup{2020} Mathematics Subject Classification}
\makeatother

\theoremstyle{plain}
\newtheorem{thm}{Theorem}[section]

\newtheorem{lem}[thm]{Lemma}
\newtheorem{prop}[thm]{Proposition}

\newtheorem{ex}[thm]{Example}
\newtheorem{de}[thm]{Definition}

\begin{document}

\title{The spectral eigenvalues of a class of  product-form self-similar spectral measure}

\author{Xiao-Yu Yan}
\author{{Wen-Hui Ai}$^{~\mathbf{*}}$}

\address{Key Laboratory of Computing and Stochastic Mathematics (Ministry of Education),
School of Mathematics and Statistics, Hunan Normal University, Changsha, Hunan 410081, P.R. China}
\email{xyyan1103@163.com}
\email{awhxyz123@163.com}

\thanks{The research is supported in part by the NNSF of China (Nos.12371072 and 12201206), the Hunan Provincial NSF (No. 2024JJ6301).\\
$^*$Corresponding author.}

\begin{abstract}
Let $\mu_{M,D}$ be the self-similar measure generated by the positive integer $M=RN^q$ and the product-form digit set $D=\{0,1,\dots,N-1\}\oplus N^{p_1}\{0,1,\dots,N-1\}\oplus \cdots \oplus N^{p_s}\{0,1,\dots,N-1\}$, where $R>1$, $N>1$, $q$, $p_i(1\leq i\leq s)$ are positive integers with $gcd(R,N)=1$ and $p_1<p_2<\cdots<p_s<q$. In this paper, we first show that $\mu_{M,D}$ is a spectral measure with a model spectrum $\Lambda$. Then we completely settle two types of spectral eigenvalue problems for  $\mu_{M,D}$. On the first case, for a real $t$, we give a necessary and sufficient condition under which $t\Lambda$ is also a spectrum of $\mu_{M,D}$. On the second case, we characterize all possible real numbers $t$ such that there exists a countable set $\Lambda'\subset \mathbb{R}$ such that $\Lambda'$ and $t\Lambda'$ are both spectra of $\mu_{M,D}$.
\end{abstract}

\maketitle

\section{Introduction}
\numberwithin{equation}{section}

\maketitle

Let $\mu$ be a compactly supported Borel probability measure on $\mathbb{R}.$ We say that $\mu$ is a spectral measure if there exists a countable discrete set $\Lambda\subset\mathbb{R}$ such that the set of exponential functions $E(\Lambda):=\{e^{2\pi i\lambda x}:\lambda\in\Lambda\}$ forms an orthonormal basis for $L^2(\mu)$, and then $\Lambda$ is called a spectrum of $\mu$. If a spectral measure $\mu$ is the Lebesgue measure on a measurable set $\Omega,$ then we say that $\Omega$ is a spectral set.

The  research of spectral measures was originated from Fuglede \cite{F74}, whose famous conjecture asserted that the normalized Lebesgue measure restricting on a domain $\Omega \subset \mathbb{R}^d$ is a spectral measure if and only if $\Omega$ is a translational tile, that is, there exists a discrete set $\Gamma\subset \mathbb{R}^d$ such that $\bigcup_{\gamma\in\Gamma}(\Omega+\gamma)$ covers $\mathbb{R}^d$ without overlaps and up to a zero set of Lebesgue measure. The conjecture was disproved eventually by Tao and others in the cases of three and higher dimensions \cite{KM2006,KM06,T04}. However, the situation in dimensions $d\leq 2$ is much less understood. In 1998, Jorgensen and Pedersen \cite{JP98} studied the spectral property of the self-similar measures and constructed the first singular, non-atomic spectral measure $\mu_{4,\{0,2\}}$. Their construction is based on a scale 4 Cantor set, where the first and third intervals are kept and the other two are discarded, i.e., $\mu_{4,\{0,2\}}$ is the standard one-fourth Cantor measure. In the same paper, they also showed that $\Lambda_{4,2}$ is a spectrum of $\mu_{4,\{0,2\}}$, where
\begin{eqnarray}\label{1.1}
\Lambda_{4,2}=\{0,1\}+4\{0,1\}+4^2\{0,1\}+\cdots\quad(\text{all finite sums}).
\end{eqnarray}

Jorgensen and Pedersen's example opened up a new area of research and many other examples of singular measures which admit orthonormal bases have been constructed including self-similar/self-affine/Moran fractal measures (see \cite{S00,LW02,D12,DHL13,DHL14,AH14,DL15,DHL19,AFL19,LMW22} and so on).
At the same time, many differences between singular spectral measures and absolutely continuous spectral measures have been discovered. For example, a typical difference is that there is only one spectrum $\Lambda=\mathbb{Z}$ containing 0 for $L^2([0,1]),$ while there are uncountable spectra (up to translations) for $L^2(\mu_{4,\{0,2\}})$ \cite{DHS09,FHW18}. Moreover, for the spectrum $\Lambda_{4,2}$ given by \eqref{1.1}, there exist infinitely many $t$ such that each $t\Lambda_{4,2}$ is also a spectrum of $\mu_{4,\{0,2\}}$ \cite{DH16,DJ12}. Motivated by the exotic facts, spectral eigenvalue problem was widely concerned and was proposed by Fu, He and Wen more systematically \cite{FHW18}.

\textbf{Spectral eigenmatrix problems.} For a given singular spectral measure $\mu$, generally speaking, there are two basic problems as follows.

\textbf{Problem 1.} Let $\Lambda$ be a spectrum of $\mu.$ Find all real matrices $R$ (or real numbers $r$) such that $R\Lambda$ (or $r\Lambda$) is also a spectrum of $\mu.$ In this case, we call $R$ (or $r$) a spectral eigenmatrix (or eigenvalue) of the first type of $\mu$.

\textbf{Problem 2.} Find all real matrices $R$ (or real numbers $r$) for which there exists a set $\Lambda$ such that both $\Lambda$ and $R\Lambda$ (or $r\Lambda$) are spectra of $\mu.$ In this case, we call $R$ (or $r$) a spectral eigenmatrix (or eigenvalue) of the second type of $\mu.$

In $\mathbb{R}$, there are several authors studying the spectral eigenvalue problem of $\mu_{4,\{0,2\}}$ with respect to the spectrum $\Lambda_{4,2}$ \cite{Dai16,DH16,DJ12,JKS11,LW02}. For example, Dutkay et al \cite{DJ12} proved that $5^k\Lambda_{4,2}$ is a spectrum of $\mu_{4,\{0,2\}}$ for all $k\in\mathbb{N}$. However, it is difficult to express or characterize all the eigenvalues $r$ clearly even in this simplest case. He et al \cite{HTW19,W19,LL21,LLW22} concentrated on the spectral eigenvalue problems of $\mu_{M, D}$ with cardinality $\#D=3$. Recently, many people (see \cite{A21,Dai16,LX17,FH17,FHW18,WDA19,WW20,W19,WZ18,ZAC23,CLL23}) considered the spectral eigenvalue problems of self-similar measures with consecutive digits $\{0,1,2,\dots,N-1\}(N\in\mathbb{N})$. Li, Wu \cite{LW22} and Lu et al \cite{LWZ23} first give the spectral eigenvalues of self-similar measures with product-form digits. On the other hand, there are few results of spectral eigenmatrix problems in dimension two and higher (see \cite{ADH22,LA23,LTW23,CL23}).

The importance of the spectral eigenvalue problems is that it is closely related to the convergence problem of Fourier expansion \cite{S00,S06,DHS14,LMW22,FTW22,PA23}. For example, Strichartz \cite{S00,S06} proved that for any continuous function $f$ on $\mathbb{R}$, its Fourier expansion with respect to the model spectrum $\Lambda_{4,2}$ in \eqref{1.1} of $\mu_{4,\{0,2\}}$ is uniformly convergent; But Dutkay et al \cite{DHS14} showed that there exists a continuous function $f$ on $\mathbb{R}$ whose Fourier expansion with respect to the spectrum $17\Lambda_{4,2}$ of $\mu_{4,\{0,2\}}$ is divergent at 0. Moreover, Jorgensen et al \cite{JKS14} found a connection between spectral eigenvalue problem and ergodic theory.

In 2022, Li and Wu \cite{LW22} studied the spectral eigenvalue problems of the self-similar measure $\mu_{R,B}$ generated by the integer $R=N^q\geq2$ and the product-form digit set $B=\{0,1,\dots,N-1\}\oplus N^p \{0,1,\dots,N-1\}$, where the integers $q>p\geq1$ and $N\geq2$. They completely settle two types of spectral eigenvalue problems for $\mu_{R,B}$. Motivated by their work, we focus on the spectrality and spectral eigenvalue problems of a class of self-similar measures $\mu_{M,D}$ generated by the positive integer $M=RN^q\geq2$ and the product-form digit set
$$D=\{0,1,\dots,N-1\}\oplus N^{p_1}\{0,1,\dots,N-1\}\oplus \cdots \oplus N^{p_s}\{0,1,\dots,N-1\}.$$
Define the iterated function system (IFS)
$$\tau_d(x)=M^{-1}(x+d) \ (x\in \mathbb{R},\ d\in D).$$
By Hutchinson's theorem \cite{H81}, there exists a unique Borel probability measure $\mu_{M,D}$ with compact support $T_{M,D}$ such that
\begin{eqnarray}\label{1.2}
\mu_{M,D}(E)=\frac{1}{\#D} \sum_{d\in D} \mu_{M,D}(\tau^{-1}_d(E))
\end{eqnarray}
for any Borel set $E\subset \mathbb{R}$ and $T_{M,D}$ is the unique compact set satisfying
\begin{eqnarray*}
T_{M,D}=\bigcup_{d\in D}\tau_d(T_{M,D}).
\end{eqnarray*}

Our main results of this paper are as follows.

\begin{thm}\label{T1.1}
Let $M=RN^q$ and $D=\{0,1,\dots,N-1\}\oplus N^{p_1}\{0,1,\dots,N-1\}\oplus \cdots \oplus N^{p_s}\{0,1,\dots,N-1\}$, where $R>1$, $N>1$, $q$, $p_i$ are positive integers  for all $1\leq i\leq s$ with $gcd(R,N)=1$ and $p_1<p_2<\cdots<p_s<q$. Then the self-similar measure $\mu_{M,D}$ defined as in \eqref{1.2} is a spectral measure with a spectrum
\begin{eqnarray}\label{1.3}
\Lambda=\left\{\sum_{j=1}^{k}(RN^q)^{j-1}l_j:k\geq1 \text{ and all }  l_j\in{L}\right\},
\end{eqnarray}
where
$$L=RN^q(\frac{1}{N}\{0,1,\dots,N-1\}\oplus\frac{1}{N^{p_1+1}}\{0,1,\dots,N-1\}\oplus\cdots\oplus\frac{1}{N^{p_s+1}}\{0,1,\dots,N-1\}).$$
\end{thm}

For the spectral measure $\mu_{M,D}$ defined as in Theorem \ref{T1.1}, we completely settle two types of spectral eigenvalue problems for  $\mu_{M,D}$.

\begin{thm}\label{T1.2}
With the notations of Theorem \ref{T1.1}, let $\Lambda$ be defined in \eqref{1.3}. Suppose that t is a real number, then $t\Lambda$ is a spectrum of $\mu_{M,D}$ if and only if $t$ is an integer with $gcd(t,N)=1$ and there is no $I=i_1i_2{\dots}i_n(I\neq{0^n})$ with all $i_j\in L$ such that
\begin{eqnarray}\label{1.4}
((RN^q)^n-1)|t(i_1+(RN^q)i_2+\cdots+(RN^q)^{n-1}i_n)
\end{eqnarray}
for all $n\leq\#\{T(RN^q,tL)\cap\mathbb{Z}\}$, where
$$
T(RN^q,tL)=\left\{\sum_{j=1}^{\infty}(RN^q)^{-j}c_j:c_j\in{tL} \text{ for all } j\geq1\right\}.
$$
\end{thm}

\begin{thm}\label{T1.3}
For a real number $t$, there exists a countable set $\Lambda'\subset \mathbb{R}$ such that $\Lambda'$ and $t\Lambda'$ are both spectra of $\mu_{M,D}$ if and only if $t=\frac{t_1}{t_2}$ for some $t_1,t_2\in\mathbb{Z}$, where $gcd(t_1,t_2)=1$ and $gcd(t_i,N)=1$ for i=1,2.
\end{thm}

The paper is structured as follows. Section \ref{2} is devoted to summarize some basic results and prove Theorem \ref{T1.1}. Section \ref{3} provides the proof of Theorem \ref{T1.2}. In Section \ref{4}, we will mainly prove Theorem \ref{T1.3}. In Section \ref{5}, we give several examples to illustrate Theorem \ref{T1.2}.

\section{Preliminaries and proof of Theorem \ref{T1.1}}\label{2}

\maketitle
In this section, we introduce some basic results that will be used in our proofs. In particular, we will prove Theorem \ref{T1.1}.

For the self-similar measure $\mu_{M,D}$ defined in \eqref{1.2}, it can also be expressed as an infinite convolution product
\begin{eqnarray*}
\mu_{M,D}&=& \delta_{M^{-1}D}*\delta_{M^{-2}D}*\delta_{M^{-3}D}*\cdots,
\end{eqnarray*}
where $\delta_{E}=\frac{1}{\#E}\sum\limits_{e\in E}\delta_e$, $\delta_e$ is the Dirac measure at the point $e\in E$ and the convergence is in weak sense.
The Fourier transform of $\mu_{M,D}$ is defined as usual,
\begin{eqnarray}\label{2.1}
\hat{\mu}_{M,D}(\xi)={\int}e^{2\pi{i}\langle \xi,x\rangle}d\mu_{M,D}(x)=\prod_{k=1}^{\infty}m_{D}(M^{-k}\xi), \quad \xi\in\mathbb{R},
\end{eqnarray}
where $m_D(x)$ is the Mask polynomial of $D$, which is defined by
$$m_D(x)=\hat{\delta}_D(x)=\frac{1}{\#D}\sum\limits_{d\in{D}}e^{2\pi{i}\langle d,x \rangle}, \quad x\in\mathbb{R}.$$
Let
$$\mu_n=\delta_{M^{-1}D}*\delta_{M^{-2}D}*\cdots*\delta_{M^{-n}D} $$
for $n\geq1$. Then
$$\mu_{M,D}=\mu_n*(\mu_{M,D}\circ M^n),$$
which yields that
$$\hat{\mu}_{M,D}(\xi)=\hat{\mu}_n(\xi)\hat{\mu}_{M,D}(M^{-n}\xi).$$
Let $\mathcal{Z}(f):=\{\xi:f(\xi)=0\}$ be the zero set of $f$. From \eqref{2.1}, it is easy to see that
$$\mathcal{Z}(\hat{\mu}_{M,D})=\bigcup_{k=1}^{\infty}M^k\mathcal{Z}(m_D).$$
\begin{lem}
$\mathcal{Z}(\hat{\mu}_{M,D})=\frac{RN^q}{N^{p_s+1}}\bigcup_{k=0}^{\infty}(RN^q)^k(\bigcup_{j=0}^{s}\{N^{p_s-p_j}l:N\nmid l\})$, where $p_0=0$.
\end{lem}
\begin{proof}
For any positive integer $m$, we have
\begin{eqnarray*}
\hat{\mu}_{M,D}(\xi)&=&\prod_{k=1}^{m}(\frac{1}{N^{s+1}}\sum\limits_{d\in D}e^{2\pi{M^{-k}d\xi i}})\hat{\mu}_{M,D}(M^{-m}\xi){}
\nonumber\\&=&\prod_{k=1}^{m}(\frac{1}{N^{s+1}}\sum_{d_1=0}^{N-1}\sum_{d_2=0}^{N-1}\cdots\sum_{d_{s+1}=0}^{N-1}e^{2\pi{M^{-k}(d_1+N^{p_1}d_2+\cdots+N^{p_s}d_{s+1})\xi i}})\hat{\mu}_{M,D}(M^{-m}\xi){}
\nonumber\\&=&\prod_{k=1}^{m}(\frac{1}{N^{s+1}}\prod_{j=1}^{s+1}(\sum_{d_j=0}^{N-1}e^{2\pi{M^{-k}N^{p_{j-1}}d_j\xi i}}))\hat{\mu}_{M,D}(M^{-m}\xi).
\end{eqnarray*}
Since $\hat{\mu}_{M,D}(M^{-m}\xi)\rightarrow 1$ as $m \rightarrow \infty$, $\hat{\mu}_{M,D}(\lambda)=0$ if and only if there exist $k\in \mathbb{N}$ and $j\in\{1,2,\dots,s+1\}$ such that
$$\sum_{d_j=0}^{N-1}(e^{2\pi M^{-k}N^{p_{j-1}}\lambda i})^{d_j}=0.$$
We see that the above equation is equivalent to
$$1-(e^{2\pi M^{-k}N^{p_{j-1}}\lambda i})^N=0,\quad e^{2\pi M^{-k}N^{p_{j-1}}\lambda i}\neq1.$$
Hence, $\hat{\mu}_{M,D}(\lambda)=0$ if and only if there exist integers $l\in \mathbb{Z}$ and $k\in \mathbb{N}$ such that $2\pi M^{-k}N^{p_{j-1}+1}\lambda =2\pi l$ and $M^{-k}N^{p_{j-1}}\lambda$ is not an integer, i.e., $\lambda=\frac{(RN^q)^kl}{N^{p_{j-1}+1}}$ with $l\in\mathbb{Z}\setminus N\mathbb{Z}$, $k \in \mathbb{N}$ and  $j\in\{1,2,\dots,s+1\}$. The conclusion follows.
\end{proof}

It is clear that a discrete set $\Lambda$ such that $E_{\Lambda}=\{e^{2\pi{i}\langle \lambda,x \rangle}:\lambda\in \Lambda\}$ is an orthonormal family for $L^2(\mu)$ if and only if
$$(\Lambda-\Lambda)\backslash \{0\}\subset\mathcal{Z}(\hat{\mu}).$$
In this case, we also say that $\Lambda$ is a bi-zero set of $\mu$. Since the bi-zero set is invariant under translation, it will be convenient to assume that $0\in \Lambda$, and thus $\Lambda \subset \Lambda-\Lambda$.

Next we introduce the Hadamard triple, which play a major role in the study of spectral measures.

\begin{de}[Hadamard triple]
Let $M\in M_n(\mathbb{Z})$ be a $n\times n$ expansive matrix $($i.e., all eigenvalues have modulus strictly greater than 1$)$ with integer entries. Let $D,C\subset \mathbb{Z}^n$ be two finite subsets of integer vectors with $\#D=\#C\geq2$. We say that the system $(M,D,C)$ forms a Hadamard triple $($or $(M^{-1}D,C)$ is a compatible  pair$)$ if the matrix
$$H=\frac{1}{\sqrt{\#D}}[e^{2\pi{i}\langle M^{-1}d,c \rangle}]_{d\in D,c\in C}$$
is unitary, i.e., $H^*H=I$, where $H^*$ denotes the conjugated transposed matrix of $H$.
\end{de}

{\L}aba and Wang  \cite{LW02} proved that the Hadamard triple guarantees the spectrality of the self-similar measure $\mu_{M,D}$ in $\mathbb{R}$.

\begin{thm}\cite[Theorem 1.2]{LW02}\label{T2.3}
Let $N\in \mathbb{Z}$ with $|N|>1$ and let $\mathcal{D}$ be a finite set of integers. Let $\mathcal{J} \subset\mathbb{Z}$ such that $0\in\mathcal{J} $ and $(\frac{1}{N}\mathcal{D},\mathcal{J} )$ is a compatible pair. Then the self-similar measure $\mu_{N,\mathcal{D}}$ is a spectral measure. If moreover $gcd(\mathcal{D}-\mathcal{D})=1$ and $0\in\mathcal{J} \subseteq[2-|N|,|N|-2]$, then $\Lambda(N,\mathcal{J})$ is a spectrum of $\mu_{N,\mathcal{D}}$, where
$$\Lambda(N,\mathcal{J})=\left\{\sum_{j=0}^{k}a_jN^j:k\geq1 \text{ and all }  a_j\in \mathcal{J} \right\}.$$
\end{thm}

Recently, Dutkay et al \cite{DHL19} extended the result to the self-affine measure $\mu_{M,D}$ in higher dimension.

\begin{thm}\cite[Theorem 1.3]{DHL19}
Let $(R,B,L)$ be a Hadamard triple with $R\in M_d(\mathbb{Z})$ and  $B,L\subset \mathbb{Z}^d$. Then the self-affine measure $\mu(R,B)$ is spectral.
\end{thm}

We record the following equivalent conditions, whose proof is now standard, so it will be omitted.

\begin{lem}\label{T2.5}
The following conditions are equivalent.

$(\romannumeral1)$ $(N,\mathcal{D},\mathcal{L})$ forms a Hadamard triple.

$(\romannumeral2)$ $\delta_{\frac{\mathcal{D}}{N}}$ is a spectral measure with spectrum $\mathcal{L}$.
\end{lem}

In order to prove Theorem \ref{T1.1}, we need the following lemma.

\begin{lem}\label{T2.6}
If $R>1$, $N>1$, $q$, $p_i$ are positive integers  for all $1\leq i\leq s$ with gcd$(R,N)=1$ and $p_1<p_2<\cdots<p_s<q$, then
$(M,D,L)$ forms a Hadamard triple,
where $$M=RN^q, D=\{0,1,\dots,N-1\}\oplus N^{p_1}\{0,1,\dots,N-1\}\oplus \cdots \oplus N^{p_s}\{0,1,\dots,N-1\}$$ and  $$L=RN^q(\frac{1}{N}\{0,1,\dots,N-1\}\oplus\frac{1}{N^{p_1+1}}\{0,1,\dots,N-1\}\oplus\cdots\oplus\frac{1}{N^{p_s+1}}\{0,1,\dots,N-1\}).$$
\end{lem}

\begin{proof}
By Lemma \ref{T2.5}, it is enough to prove the result if $L$ is a spectrum of the measure $\delta_{M^{-1}D}$. We first claim that $L$ is an orthogonal set of $\delta_{M^{-1}D}$. Indeed, it is easily seen that $$\mathcal{Z}(\hat{\delta}_{M^{-1}D})=M\mathcal{Z}(\hat{\delta}_D)=\frac{RN^q}{N^{p_s+1}}(\bigcup_{j=0}^{s}\{N^{p_s-p_j}l:N\nmid l\}),$$
 where $p_0=0.$ For any two distinct elements $l,l'\in L$, write
 $$l=RN^q\sum_{i=1}^{s+1}\frac{l_i}{N^{p_{i-1}+1}} \text{ and }  l'=RN^q\sum_{i=1}^{s+1}\frac{l_i'}{N^{p_{i-1}+1}},$$
where $l_i,l_i'\in \{0,1,\dots,N-1\}$ for all $i\in\{1,2,\dots,s+1\}$. Then, we have
$$l-l'=\frac{RN^q}{N^{p_s+1}}\sum_{i=1}^{s+1}N^{p_s-p_{i-1}}(l_i-l_i').$$

Case 1: $l_{s+1}\neq l_{s+1}'$. Then $N\nmid (l_{s+1}-l_{s+1}')$ and $l-l'\in \mathcal{Z}(\hat{\delta}_{M^{-1}D})$.

Case 2: $l_{s+1} = l_{s+1}', l_s\neq l_s'$. We have
$$l-l'=\frac{RN^q}{N^{p_s+1}}(N^{p_s-p_{s-1}}(\sum_{i=1}^{s}N^{p_{s-1}-p_{i-1}}(l_i-l_i')))$$
and $N\nmid (l_s-l_s')$. Hence $l-l'\in \mathcal{Z}(\hat{\delta}_{M^{-1}D})$.

$\vdots$

Case {s+1}: $l_{s+1} = l_{s+1}', \cdots ,l_2 = l_2' ,l_1 \neq l_1'.$ Then $l-l'=\frac{RN^q}{N^{p_s+1}}N^{p_s}(l_1-l_1')$. Similarly as above, we have that $N\nmid (l_1-l_1')$ and thus $l-l'\in \mathcal{Z}(\hat{\delta}_{M^{-1}D})$.

Consequently, the claim follows. It is also a simple matter to the cardinality of the set $L$ is equal to the dimension of the Hilbert space $L^2({\delta}_{M^{-1}D})$ and thus $L$ is a spectrum of ${\delta}_{M^{-1}D}$. This completes the proof of Lemma \ref{T2.6}.
\end{proof}

The proof of Theorem \ref{T1.1} is now immediate.

\begin{proof}[{Proof of Theorem \ref{T1.1}}]
We know $(M^{-1}D,L)$ is a compatible pair by Lemma \ref{T2.6}. On the other hand, it is easy to check that the conditions of Theorem \ref{T2.3} holds for the measure $\mu_{M,D}$. Therefore, we obtain a spectrum $\Lambda$ of the measure $\mu_{M,D}$  which has the following form
$$\Lambda=\left\{\sum_{j=1}^{k}(RN^q)^{j-1}l_j:k\geq1 \text{ and all }  l_j\in{L}\right\}.$$
The proof is complete.
\end{proof}

The following lemma is an effective method to illustrate that a countable set $\Lambda$ cannot be a spectrum of a measure $\mu$, which will be used in Lemma \ref{T3.2}.

\begin{lem}\cite[Lemma 2.2]{DHL14}\label{T2.7}
Let $\mu=\mu_0*\mu_1$ be the convolution of two probability measures $\mu_i,i=0,1,$ and they are not Dirac measures. Suppose that $\Lambda$ is a bi-zero set of $\mu_0$, then $\Lambda$ is also a bi-zero set of $\mu$, but cannot be a spectrum of $\mu$.
\end{lem}

\section{Proof of Theorem \ref{T1.2}}\label{3}
\maketitle
In this section, we prove Theorem \ref{T1.2}.  For that purpose, we need several lemmas.

\begin{lem}\label{T3.1}
The following conditions are equivalent.

$(\romannumeral1)$ \ gcd$(t,N)=1$.

$(\romannumeral2)$ $t\mathbb{Z}$ contains a complete set of the set $\mathbb{Z}\setminus N\mathbb{Z}$.
\end{lem}
\begin{proof}
$``(\romannumeral1)\Rightarrow(\romannumeral2)"$ Suppose on the contrary that there exists $d\in\{0,1,\ldots,N-1\}$ such that $t\mathbb{Z}\cap(d+N\mathbb{Z})=\emptyset$. It follows that $tk\notin d+N\mathbb{Z}$ for all integers $k$. Since gcd$(t,N)=1$, then there exist two integers $u,v$ such that $tu+Nv=1$. Hence $tud=d-Nvd\in d+N\mathbb{Z}$, a contradiction.

$``(\romannumeral2)\Rightarrow(\romannumeral1)"$ Suppose gcd$(t,N)=d>1$, it is easy to see $t^{-1}(d-1+N\mathbb{Z})\cap\mathbb{Z}\neq\emptyset$. We can choose $k\in t^{-1}(d-1+N\mathbb{Z})\cap\mathbb{Z}$ such that gcd$(tk,d)=1$. It is impossible.
\end{proof}

For convenience, let $B=\{0,1,\dots,N-1\}\oplus N^{p_s-p_{s-1}}\{0,1,\dots,N-1\}\oplus\cdots\oplus N^{p_s-p_0}\{0,1,\dots,N-1\}$ with $p_0=0.$ We have $ L=\frac{RN^q}{N^{p_s+1}}B$.

\begin{lem}\label{T3.2}
  If $t\Lambda$ is a spectrum of $\mu_{M,D}$, then $t\in\mathbb{Z}$ and gcd$(t,N)=1$.
\end{lem}
\begin{proof}
The proof will be divided into four steps.

\textbf{Step 1.} We prove that $t\in\mathbb{Z}$. It is evident that $t$ must be an integer by the fact
$$t\Lambda\setminus\{0\}\subset\mathcal{Z}(\hat{\mu}_{M,D})=\frac{RN^q}{N^{p_s+1}}\bigcup_{k=0}^{\infty}(RN^q)^k(\bigcup_{j=0}^{s}\{N^{p_s-p_j}l:N\nmid l\}),$$
where $p_0=0$,
$$t\Lambda=\left\{t\frac{RN^q}{N^{p_s+1}}\sum_{j=1}^{k}(RN^q)^{j-1}b_j:k\geq1 \text{ and all } b_j\in B\right\}$$
with $gcd(\sum_{j=1}^{k}(RN^q)^{j-1}B)=1$.

\textbf{Step 2.} We claim that $t\Lambda\cap\Lambda_0\neq\emptyset$, where $\Lambda_0=\frac{RN^q}{N^{p_s+1}}(\bigcup_{j=0}^{s}\{N^{p_s-p_j}l:N\nmid l\})$. Suppose, contrary to our claim, that
$$t\Lambda\subseteq\{0\}\cup\frac{1}{N^{p_s+1}}\bigcup_{k=2}^{\infty}(RN^q)^{k}(\bigcup_{j=0}^{s}\{N^{p_s-p_j}l:N\nmid l\}).$$
For any two distinct elements $t\lambda_1,t\lambda_2\in t\Lambda\setminus\{0\}$, write $t\lambda_1=\frac{1}{N^{p_s+1}}(RN^q)^{k_1}N^{p_s-p_{j_1}}l_1$ and $t\lambda_2=\frac{1}{N^{p_s+1}}(RN^q)^{k_2}N^{p_s-p_{j_2}}l_2$, where integers $k_i\geq2$, $j_i\in\{0,1,\dots,s\}$ and $N\nmid l_i$ for $i=1,2$. Then, it is easy to check
$$t\lambda_2-t\lambda_1\in \frac{1}{N^{p_s+1}}\bigcup_{k=2}^{\infty}(RN^q)^{k}(\bigcup_{j=0}^{s}\{N^{p_s-p_j}l:N\nmid l\}),$$
where $k_1\neq k_2$ or $j_1\neq j_2$. For this case, $k_1=k_2,j_1=j_2,l_1\neq l_2,$ we have
$$t\lambda_2-t\lambda_1=\frac{1}{N^{p_s+1}}(RN^q)^{k_1}N^{p_s-p_{j_1}}(l_2-l_1)\notin \Lambda_0,$$
since $N|\frac{(RN^q)^{k_1-1}}{N^{p_{j_1}}}(l_2-l_1)$. It follows from $t\Lambda$ is a spectrum of the measure $\mu_{M,D}$ that
$$t\lambda_2-t\lambda_1\in \frac{1}{N^{p_s+1}}\bigcup_{k=2}^{\infty}(RN^q)^{k}(\bigcup_{j=0}^{s}\{N^{p_s-p_j}l:N\nmid l\}).$$
Hence, we have
$$(t\Lambda-t\Lambda)\subseteq\{0\}\cup\frac{1}{N^{p_s+1}}\bigcup_{k=2}^{\infty}(RN^q)^{k}(\bigcup_{j=0}^{s}\{N^{p_s-p_j}l:N\nmid l\}), $$
which implies that $t\Lambda$ is an orthogonal set of the measure $v=\delta_{M^{-2}D}*\delta_{M^{-3}D}*\cdots$. Note that $\mu_{M,D}=\delta_{M^{-1}D}*v$. By Lemma \ref{T2.7}, $t\Lambda$ is not a spectrum of $\mu_{M,D}$, which is a contradiction.

\textbf{Step 3.} Our goal is to prove that $c^{-1}(t\Lambda\cap\Lambda_0)\cup\{0\}$ contains a complete residue system (mod $N$), where $c=\frac{RN^q}{N^{p_s+1}}$. Suppose, on the contrary, that there exists $n_0\in\{0,1,\dots,N-1\}$ such that $$c^{-1}(t\Lambda\cap\Lambda_0)\cap(n_0+N\mathbb{Z})=\emptyset.$$
It follows that $\lambda-cn_0\in c(\mathbb{Z}\setminus N\mathbb{Z})\subseteq\mathcal{Z}(\hat{\mu}_{M,D})$ for all $\lambda\in t\Lambda\cap\Lambda_0$. This implies that $(t\Lambda\cap\Lambda_0)\cup\{cn_0\}$ is an orthogonal set of $\mu_{M,D}$. On the other hand, for any $\lambda\in t\Lambda\setminus \Lambda_0$, write $\lambda=c(RN^q)^kN^{p_s-p_j}l$ with $k\in \mathbb{N}$, $j\in\{0,1,\dots,s\}$ and $N\nmid l$ or 0. We have $$\lambda-cn_0=c((RN^q)^kN^{p_s-p_j}l-n_0) \text{ or } -cn_0.$$
In this two cases, they all belong to $c(\mathbb{Z}\setminus N\mathbb{Z})\subseteq\mathcal{Z}(\hat{\mu}_{M,D})$. Then $cn_0$ is orthogonal to all elements in $t\Lambda$. This contradicts the fact that $t\Lambda$ is a spectrum of $\mu_{M,D}$ and the claim follows.

\textbf{Step 4.} We prove that gcd$(t,N)=1$. Suppose otherwise, gcd$(t,N)>1$. Then $t\mathbb{Z}$ does not contain a complete set of the set $\mathbb{Z}\setminus N\mathbb{Z}$ by Lemma \ref{T3.1}. Note that $c^{-1}(t\Lambda\cap\Lambda_0)\subseteq c^{-1}t\Lambda\subseteq t\mathbb{Z}$. This leads to a contradiction.
\end{proof}

Similar to the proof of Lemma \ref{T2.6}, we have the following lemma.

\begin{lem}\label{T3.3}
  If gcd$(t,N)=1$, then $((RN^q)^{-1}D,tL)$ is compatible.
\end{lem}

If $(\frac{1}{N}\mathcal{D},\mathcal{J})$ is a compatible pair, then $\mu_{N,\mathcal{D}}$ is a spectral measure\cite{DHL19}. The following theorem gives a necessary and sufficient condition under which $(\mu_{N,\mathcal{D}},\Lambda(N,\mathcal{J}))$ is not a spectral pair.

\begin{thm}\cite[Theorem 1.3]{LW02}\label{T3.4}
Let $N\in \mathbb{Z}$ with $|N|>1$ and $\mathcal{D}\subset \mathbb{Z}$ with $0\in \mathcal{D}$ and $gcd(\mathcal{D})=1$. Let $\mathcal{J} \subset \mathbb{Z}$ with $0\in \mathcal{J} $ such that $(\frac{1}{N}\mathcal{D},\mathcal{J})$ is a compatible pair. Then $(\mu_{N,\mathcal{D}},\Lambda(N,\mathcal{J}))$ is NOT a spectral pair if and only if there exist $s_j^*\in \mathcal{J}$ and nonzero integers $\eta_j$, $0\leq j \leq m-1$, such that $\eta_{j+1}=N^{-1}(\eta_j+s_j^*)$ for all $0\leq j \leq m-1(with \ \eta_m:=\eta_0 \ and \ s_m^*:=s_0^*)$.
\end{thm}

For the sake of convenience, we introduce some notations. Let $A\subset \mathbb{Z}$, denote by
$$\Sigma_{A}^{n}=\{i_1i_2\dots i_n:i_k\in A,1\leq k \leq n\}, \ \Sigma_{A}^{*}=\bigcup_{n=1}^{\infty}\Sigma_{A}^{n},$$
and
$$\Sigma_{A}^{\infty}=\{i_1i_2\ldots:i_k\in A,k\geq1\},$$
the set of all words with length $n$, the set of all finite words and the set of all infinite words respectively. For any $I\in \Sigma_{A}^{*},J\in \Sigma_{A}^{*}\cup \Sigma_{A}^{\infty}$, we denote by $IJ$ the natural concatenation of $I$ and $J$. Furthermore, we adopt the notations $I^{\infty}=III\ldots, I^k=\underbrace{I\ldots I}_k$ for $I\in \Sigma_{A}^{*}$. Let $I=i_1i_2\ldots\in \Sigma_{A}^{\infty}$. We define $I|_k=i_1i_2\ldots i_k$ for $k\geq1$ and $I_{n,m}=i_ni_{n+1}\ldots i_{m-1}$ for $1\leq n<m\leq \infty$. With the help of the preceding lemmas and theorem, we can now prove the following lemma which is important for us to prove Theorem \ref{T1.2}.

\begin{lem}\label{T3.5}
Let $gcd(t,N)=1$. Then the following statements are equivalent:

$(\romannumeral1)$ \ $t\Lambda$ is not a spectrum of $\mu_{M,D}$;

$(\romannumeral2)$ There exists a word $I=i_1\ldots i_m\in\Sigma_{B}^{*}$ with $I\neq 0^m$ such that $$((RN^q)^m-1)|t\frac{RN^q}{N^{p_s+1}}(i_1+(RN^q)i_2+\cdots+(RN^q)^{m-1}i_m).$$
\end{lem}
\begin{proof}
$``(\romannumeral1)\Rightarrow(\romannumeral2)"$ Since gcd$(t,N)=1$, by Lemma \ref{T3.3}, $((RN^q)^{-1}D,tL)$ is compatible. Then by Theorem \ref{T3.4}, there exist $s_j^*\in tL$  and nonzero integers $\eta_j$ such that $\eta_{j+1}=\frac{1}{RN^q}(\eta_j+s_j^*)$ for all $0\leq j\leq m-1$ (with $\eta_m:=\eta_0$ and $s_m^*:=s_0^*$). We have
\begin{eqnarray*}
\eta_1&=&\frac{1}{RN^q}\eta_m+\frac{1}{RN^q}s_m^*{}
\nonumber\\&=&\frac{1}{(RN^q)^2}\eta_{m-1}+\frac{1}{(RN^q)^2}s_{m-1}^*+\frac{1}{RN^q}s_m^*{}
\nonumber\\&=&\dots{}
\nonumber\\&=&\frac{1}{(RN^q)^m}\eta_1+\frac{1}{(RN^q)^m}s_1^*+\frac{1}{(RN^q)^{m-1}}s_2^*+\dots+\frac{1}{RN^q}s_m^*.
\end{eqnarray*}
It gets that
\begin{eqnarray*}
((RN^q)^m-1)\eta_1&=&s_1^*+(RN^q)s_2^*+\dots+(RN^q)^{m-1}s_m^*{}
\nonumber\\&=&t\frac{RN^q}{N^{p_s+1}}(i_1+(RN^q)i_2+\dots+(RN^q)^{m-1}i_m),
\end{eqnarray*}
where $i_j=\frac{s_j^*}{t\frac{RN^q}{N^{p_s+1}}}\in B$ for $1\leq j\leq m$.

$``(\romannumeral2)\Rightarrow(\romannumeral1)"$ Suppose there exists a word $I=i_1i_2\dots i_m\in \Sigma_{B}^{*}$ with $I\neq 0^m$ such that $$((RN^q)^m-1)|t\frac{RN^q}{N^{p_s+1}}(i_1+(RN^q)i_2+\dots+(RN^q)^{m-1}i_m).$$
Let $\tau(\sigma)=t(\sigma_1+(RN^q)\sigma_2+\dots+(RN^q)^{n-1}\sigma_n)$ for $\sigma=\sigma_1\sigma_2\dots\sigma_n\in \Sigma_{B}^{*}\cup\Sigma_{B}^{\infty}$. We have $((RN^q)^m-1)|\frac{RN^q}{N^{p_s+1}}\tau(I)$. Denote $d=\frac{RN^q}{N^{p_s+1}}\frac{\tau(I)}{1-(RN^q)^m}:=\frac{RN^q}{N^{p_s+1}}d'$. It follows from gcd$((RN^q)^m-1,\frac{RN^q}{N^{p_s+1}})=1$ that $d'\in \mathbb{Z}$. Let
$$\Lambda'=\left\{\sum_{j=1}^{k}(RN^q)^{j-1}b_j:k\geq1 \text{ and all } b_j\in B\right\}.$$
Notice that $$d'=\tau(I)+d'(RN^q)^m=\tau(I)+\tau(I)(RN^q)^m+d'(RN^q)^{2m}=\tau(I^2)+d'(RN^q)^{2m}.$$
Proceeding inductively, it gets that
$$d'=\tau(I^j)+d'(RN^q)^{jm} $$
for any $j\geq1$. Since $ \tau(I^{\infty}|_n)=\tau(I^{j}|_n)$ for any $1\leq n\leq mj$, we have
\begin{align*}
\tau(I^{\infty})=&t(i_1+(RN^q)i_2+\dots+(RN^q)^{m-1}i_m+\cdots\\
&+(RN^q)^{(j-1)m}i_1+(RN^q)^{(j-1)m+1}i_2+\dots+(RN^q)^{jm-1}i_m+\cdots)\\
=&\tau(I^{\infty}|_{jm})+(RN^q)^{jm}\tau(I^{\infty})=\tau(I^{j})+(RN^q)^{jm}\tau(I^{\infty})
\end{align*}
for any $j\geq1$. Combining with $I^{\infty}\in \Sigma_{B}^{\infty}\setminus\Sigma_{B}^{*}0^{\infty}$, $d'=\tau(I^{\infty})\notin t\Lambda'$.  Therefore, $d\notin t\Lambda$.

For any $\lambda\in t\Lambda$, write $\lambda=t\frac{RN^q}{N^{p_s+1}}\lambda'$ with $\lambda'\in \Lambda'$. Next we show that $d$ is orthogonal to $\lambda$. For the $\lambda'\in \Lambda'$, there exists $J=j_1j_2\dots j_n\in\Sigma_{B}^{n}$ such that $\tau(J)=t\lambda'$. Let $s\geq1$ be the first index such that $(I^{\infty})_{s,s+1}\neq(J0^{\infty})_{s,s+1}$. Then
\begin{align*}
d-\lambda=&\frac{RN^q}{N^{p_s+1}}(d'-t\lambda')\\
=&t\frac{RN^q}{N^{p_s+1}}(RN^q)^{s-1}((I^{\infty})_{s,s+1}-(J0^{\infty})_{s,s+1}+RN^qK)
\end{align*}
for some integer $K$  and $(I^{\infty})_{s,s+1},(J0^{\infty})_{s,s+1}\in B$. We claim that $$N\nmid((I^{\infty})_{s,s+1}-(J0^{\infty})_{s,s+1}).$$ In fact, for any two distinct elements $b=b_1+N^{p_s-p_{s-1}}b_2+\cdots+N^{p_s-p_0}b_{s+1},b'=b'_1+N^{p_s-p_{s-1}}b'_2+\cdots+N^{p_s-p_0}b'_{s+1}\in B$, where $b_i,b'_i\in\{0,1,\dots,N-1\}$ for all $i\in \{1,2,\dots,s+1\}$, we have
$$b-b'=b_1-b'_1+N^{p_s-p_{s-1}}(b_2-b'_2+N^{p_{s-1}-p_{s-2}}(b_3-b'_3)+\dots+N^{p_{s-1}-p_0}(b_{s+1}-b'_{s+1})).$$ It is easy to see that $N\nmid (b-b')$ and the claim follows. It follows from gcd$(t,N)=1$ that $d-\lambda\in\mathcal{Z}(\hat{\mu}_{M,D})$. Then $t\Lambda$ is not a maximal orthogonal set and thus is not a spectrum of $\mu_{M,D}$.
\end{proof}

We have all ingredients for the proof of Theorem \ref{T1.2}.

\begin{proof}[{Proof of Theorem \ref{T1.2}}]
  The necessity follows immediately by Lemma \ref{T3.2} and  Lemma \ref{T3.5}. For the sufficiency, by Lemma \ref{T3.5} again, we need to show that there is not a word $I=i_1\dots i_m\in\Sigma_{B}^{*}$ with $I\neq 0^m$ such that
  \begin{eqnarray}\label{3.1}
  ((RN^q)^m-1)|t\frac{RN^q}{N^{p_s+1}}(i_1+(RN^q)i_2+\cdots+(RN^q)^{m-1}i_m).
  \end{eqnarray}
  By the sufficient condition, we only need to show that there is not a word $I=i_1\dots i_m\in\Sigma_{B}^{*}$ with $I\neq 0^m$ such that \eqref{3.1}  holds for all $m> l$, where $l=\#\{T(RN^q,tL)\cap\mathbb{Z}\}$. Suppose, on the contrary, that there exists $I=i_1i_2\dots i_n\in\Sigma_{B}^{n}\setminus\Sigma_{B}^{l}0^{n-l}$ for some $n>l$, such that \eqref{3.1} holds. Our next goal is to construct $I'=i_1'\dots i_r'\in\Sigma_{B}^{*}(I' \neq 0^r)$ with $r<l$ satisfying \eqref{3.1}, a contradiction.

Let $\tau(\sigma)=t(\sigma_1+(RN^q)\sigma_2+\cdots+(RN^q)^{k-1}\sigma_k)$ for $\sigma=\sigma_1\sigma_2\dots\sigma_k\in\Sigma_{B}^{*}$. For $\sigma=\sigma_1\sigma_2\dots\sigma_m$, denote $\Pi(\sigma)=\sigma_2\dots\sigma_m\sigma_1$. We claim that
\begin{equation}\label{3.2}
     \frac{\frac{RN^q}{N^{p_s+1}}\tau(\Pi^k(I))}{(RN^q)^n-1}\in T(RN^q,tL)\cap\mathbb{Z}.
\end{equation}
In fact, note that
$$\frac{\frac{RN^q}{N^{p_s+1}}\tau(\Pi^k(I))}{(RN^q)^n-1}=\sum_{i=1}^{\infty}(RN^q)^{-in}\frac{RN^q}{N^{p_s+1}}\tau(\Pi^k(I)), $$
and $\tau(\Pi^k(I))=t\sum_{k=0}^{n-1}(RN^q)^ki'_{k+1}$, where $\{i'_k\}_{k=1}^{n}$ is a rearrangement of $\{i_k\}_{k=1}^{n}$,
we have
$$\frac{\frac{RN^q}{N^{p_s+1}}\tau(\Pi^k(I))}{(RN^q)^n-1}=\sum_{i=1}^{\infty}\sum_{k=0}^{n-1}(RN^q)^{-in+k}t\frac{RN^q}{N^{p_s+1}}i'_{k+1}\in T(RN^q,tL).$$
Moreover,
\begin{align*}
\frac{\frac{RN^q}{N^{p_s+1}}\tau(\Pi(I))}{(RN^q)^n-1}=&\frac{\frac{RN^q}{N^{p_s+1}}{t(i_2+(RN^q)i_3+\cdots+(RN^q)^{n-2}i_n+(RN^q)^{n-1}i_1)}}{(RN^q)^n-1}\\
=&\frac{\frac{RN^q}{N^{p_s+1}}{t(i_1+(RN^q)i_2+\cdots+(RN^q)^{n-1}i_n+(RN^q)^{n}i_1-i_1)}}{RN^q((RN^q)^n-1)}\\
=&\frac{\frac{RN^q}{N^{p_s+1}}{(\tau(I)+t((RN^q)^n-1)i_1)}}{RN^q((RN^q)^n-1)}.
\end{align*}
Hence
 $${RN^q}\frac{RN^q}{N^{p_s+1}}\tau(\Pi(I))=\frac{RN^q}{N^{p_s+1}}(\tau(I)+t((RN^q)^n-1)i_1).$$
It is known that $((RN^q)^n-1)|\frac{RN^q}{N^{p_s+1}}\tau(I)$. It follows from gcd$(RN^q,(RN^q)^n-1)=1$ that $((RN^q)^n-1)|\frac{RN^q}{N^{p_s+1}}\tau(\Pi(I)).$ By induction, we have $\frac{\frac{RN^q}{N^{p_s+1}}\tau(\Pi^k(I))}{(RN^q)^n-1} \in \mathbb{Z}$ and the \eqref{3.2} is proved.

Since $I=\Pi^n(I)$, there exists the smallest positive integer $r$ such that $I=\Pi^r(I)$. Then $I,\Pi(I),\cdots,\Pi^{r-1}(I)$ are pairwise different. By (\ref{3.2}) and the fact that $\tau$ is injective, we have $r\leq l<n$. Moreover, $\{k\in \mathbb{Z}:\Pi^k(I)=I\}=r\mathbb{Z}$ is a subgroup of $\mathbb{Z}$, where $\Pi^{-1}(\sigma)=\sigma_n\sigma_1\dots\sigma_{n-1}$ for $\sigma=\sigma_1\sigma_2\dots\sigma_n$. Then $r|n$. Write $n=kr$ and thus $I=(I|_r)^k$. Note that
$$\tau(I)=\tau((I|_r)^k)=\tau(I|_r)\frac{1-(RN^q)^{kr}}{1-(RN^q)^r}=\tau(I|_r)\frac{1-(RN^q)^{n}}{1-(RN^q)^r},$$
we have $$\frac{\frac{RN^q}{N^{p_s+1}}\tau(I)}{(RN^q)^n-1}=\frac{\frac{RN^q}{N^{p_s+1}}\tau(I|_r)}{(RN^q)^r-1},$$ which contradicts the sufficient condition and thus by Lemma \ref{T3.5}, the proof of Theorem \ref{T1.2} is complete.
\end{proof}

\section{Proof of Theorem \ref{T1.3}}\label{4}
\maketitle
In this section, we prove Theorem \ref{T1.3}.

\begin{proof}[Proof of the necessity of Theorem \ref{T1.3}]
  Suppose that there exist $\Lambda$ and $t\Lambda$ such that both of them are spectra of $\mu_{M,D}$. Without loss of generality, we can assume that $0\in\Lambda$. Then $\Lambda,t\Lambda\subseteq\mathcal{Z}(\hat{\mu}_{M,D})\cup\{0\}$, which implies that $t$ is a rational number, say, $t=\frac{t_1}{t_2}$ with $t_1,t_2\in \mathbb{Z}\setminus\{0\}$ and gcd$(t_1,t_2)=1$. Next, we prove the necessity by contradiction.

Suppose gcd$(t_1,N)=:d_1>1$. By Lemma \ref{T3.1}, $t_1\mathbb{Z}$ does not contain a complete set of the set $\mathbb{Z} \setminus  N\mathbb{Z}$. Let $\Lambda_0=\frac{RN^q}{N^{p_s+1}}(\bigcup_{j=0}^{s}\{N^{p_s-p_j}l:N\nmid l\})$ with $p_0=0$. Then by the same idea as in Lemma \ref{T3.2}, $c^{-1}(\Lambda_0\cap t\Lambda)\cup\{0\}$ contains a complete set of the set $\mathbb{Z}\setminus N\mathbb{Z}$, where $c=\frac{RN^q}{N^{p_s+1}}$. On the other hand,
$$c^{-1}(\Lambda_0\cap t\Lambda)\cup\{0\}\subseteq c^{-1}\frac{t_1}{t_2}\Lambda\cup\{0\}\subseteq t_1\mathbb{Z},$$
where the last inclusion relation follows from the facts that $\Lambda\subseteq c\mathbb{Z}$ and $t\Lambda\subseteq c\mathbb{Z}$. This contradicts the fact that $t_1\mathbb{Z}$ does not contain a complete set of the set $\mathbb{Z}\setminus N\mathbb{Z}$.

Suppose gcd$(t_2,N):=d_2>1$. By the same idea again as in Lemma \ref{T3.2}, $c^{-1}(\Lambda_0\cap\Lambda)\cup\{0\}$ contains a complete set of the set $\mathbb{Z}\setminus N\mathbb{Z}$. Then there exists $\lambda_0\in \Lambda_0\cap\Lambda$ such that gcd$(d_2,\frac{\lambda_0}{c})=1$. Here, we can choose $\lambda_0\in c(d_2-1+N\mathbb{Z})\cap(\Lambda_0\cap\Lambda).$ It follows from gcd$(t_1,d_2)=1$ that  $t\frac{\lambda_0}{c}=\frac{t_1\frac{\lambda_0}{c}}{\frac{t_2}{d_2}d_2}\notin\mathbb{Z}$. We have $t\lambda_0\notin c\mathbb{Z}$ and thus $t\lambda_0\notin \mathcal{Z}(\hat\mu_{M,D})$. This contradicts the fact $t\Lambda$ is a spectrum of $\mu_{M,D}$.
\end{proof}

The following theorem gives a sufficient condition for an orthogonal set of $\mu_{M,D}$ to be a spectrum.

\begin{thm}\cite[Theorem 2.7,Theorem 2.8]{S00}\label{T4.1}
Given an infinite compatible tower whose compatible pairs $\{B_j,L_j\}$ and expanding matrices $\{R_j\}$ are all chosen from a finite set of compatible pairs and expanding matrices, the infinite convolution product measure
$$\mu=\mu_0*(\mu_1\circ R_1)*\cdots*(\mu_k\circ(R_k\cdots R_1))*\cdots$$
exists as a weak limit and is a compactly supported probability measure, whose Fourier transform is given by the infinite product
$$\hat{\mu}(t)=\prod_{k=0}^{\infty}\hat{\mu}_k((R_k^*)^{-1}\cdots(R_1^*)^{-1}t).$$
Moreover, the set of functions $\{e^{2\pi{i}{ \langle x,\lambda \rangle }}\}_{\lambda\in\Lambda}$ for
$$\Lambda=L_0+R_1^*L_1+\cdots+R_1^*\cdots R_k^*L_k+\cdots$$
is an orthogonal set.
Suppose that the zero set $\mathcal{Z}_k$ of the trigonometric polynomial $\hat{\mu}_k$ is separated from the set
$$(R_k^*)^{-1}\cdots(R_1^*)^{-1}L_0+(R_k^*)^{-1}\cdots(R_2^*)^{-1}L_1+\cdots+(R_k^*)^{-1}L_{k-1}$$
by a distance $\delta>0$, uniformly in k, for all large k. Then $\Lambda$ is a spectrum for $\mu$.
\end{thm}

To show the sufficiency of Theorem \ref{T1.3}, we first consider the case that $t$ is an integer with gcd$(t,N)=1.$ In this case, the scaling set $t\Lambda$ is not guaranteed also a spectrum, where $\Lambda$ is defined as in \eqref{1.3}. We have to need more choices. The key idea is that we can give a new class of spectra by the way of using infinite word in $\{-1,1\}^{\infty}$ acting on $\Lambda.$ This is first observed by Fu et al for the Bernoulli convolution in \cite{FHW18}. This provides more candidate of spectra. Precisely, we have the following proposition.

\begin{prop}
  For every $\omega\in \{-1,1\}^{\infty}$, the set
 $$ \Lambda_{\omega}=\left\{\sum_{j=1}^{k}l_j\omega_j(RN^q)^{j-1}:k\geq1 \text{ and all } l_j\in L \right\},$$
 where $L=\frac{RN^q}{N^{p_s+1}}B$ and $B=\{0,1,\dots,N-1\}\oplus N^{p_s-p_{s-1}}\{0,1,\dots,N-1\}\oplus\cdots\oplus N^{p_s-p_0}\{0,1,\dots,N-1\}$ with $p_0=0$, is a spectrum of $\mu_{M,D}$.
\end{prop}
\begin{proof}
  Fix an $\omega\in\{-1,1\}^{\infty}$ and $n\geq 1$, write
  $$T_n=\left\{(RN^q)^{-n}\sum_{i=1}^{n}l_i\omega_i(RN^q)^{i-1}:l_i\in L \text{ for } 1\leq i \leq n \right\}.$$
Notice that
 \begin{align*}
&(RN^q)^{-n}\sum_{i=1}^{n}\frac{RN^q}{N^{p_s+1}}(N-1+N^{p_s-p_{s-1}}(N-1)+\cdots+N^{p_s-p_0}(N-1))(RN^q)^{i-1}\\
=&(RN^q)^{-n}\frac{RN^q}{N^{p_s+1}}(N-1)(1+N^{p_s-p_{s-1}}+\cdots+N^{p_s-p_0})\frac{(RN^q)^n-1}{RN^q-1}\\
<&\frac{RN^q}{N^{p_s+1}}\frac{(N-1)(1+N^{p_s-p_{s-1}}+\cdots+N^{p_s-p_0})}{RN^q-1}\\
<&\frac{RN^q}{N^{p_s+1}}.
\end{align*}
Therefore, the set
$$\mathcal{Z}(\hat{\mu}_{n})=\frac{RN^q}{N^{p_s+1}}\bigcup_{k=0}^{n-1}(RN^q)^k(\bigcup_{j=0}^{s}\{N^{p_s-p_j}l:N\nmid l\})$$
is uniformly disjoint form the sets $T_n,n\geq1$. On the other hand, by Lemma \ref{T3.3}, $((RN^q)^{-1}D,\omega_kL)$ is compatible for each $\omega_k$. Then we have $\Lambda_\omega$ is spectrum of $\mu_{M,D}$ by Theorem \ref{T4.1}.
\end{proof}

The following result may be proved in much the same way as Theorem \ref{T1.2}.

\begin{thm}\label{T4.3}
  Suppose $\omega=(\omega_1\omega_2\dots\omega_r)^{\infty}\in\{-1,1\}^{\infty}$ be a periodic word with periodicity $r$. Let t be an integer number with gcd$(t,N)=1$. Then $t\Lambda_\omega$ is a spectrum of $\mu_{M,D}$ if and only if there is no $I=i_1i_2\dots i_n(I\neq 0^n)$ with all $i_j\in \Sigma_{\omega,r}$ such that
  $$((RN^q)^{nr}-1)|t(i_1+(RN^q)^ri_2+\cdots+(RN^q)^{(n-1)r}i_n)$$
for all $n\leq\#\{T((RN^q)^r,t\Sigma_{\omega,r})\cap\mathbb{Z}\},$ where $\Sigma_{\omega,r}=\{\sum_{k=1}^{r}l_k\omega_k(RN^q)^{k-1}:  l_k\in L \text{ for } 1\leq k\leq r\}$ and
$$T((RN^q)^r,t\Sigma_{\omega,r})=\left\{\sum_{j=1}^{\infty}\frac{d_j}{(RN^q)^{rj}}:d_j\in{t\Sigma_{\omega,r}} \text{ for } j\geq1\right\}.$$
\end{thm}

The following theorem is crucial to find the spectral eigenvalues of the second type of $\mu_{M,D}$.

\begin{thm}\label{T4.4}
  If $t\in \mathbb{Z}$ with gcd$(t,N)=1$, then there exists an infinite word $\omega=\omega_1\omega_2\ldots\in \{-1,1\}^{\infty}$ such that $t\Lambda_{\omega}$ is a spectrum of $\mu_{M,D}$.
\end{thm}
\begin{proof}
  There exists $\omega=(\omega_1\omega_2\ldots\omega_r)^{\infty}$ with sufficiently large periodicity $r$ such that for any $I=i_1i_2\ldots i_n(I\neq 0^n)$ with all $i_j\in \Sigma_{\omega,r}$, one has $((RN^q)^{nr}-1)> |t(i_1+(RN^q)^ri_2+\cdots+(RN^q)^{(n-1)r}i_n)|$ (to achieve it, we can require that the last few items of the repetend take values -1) and thus
  $$((RN^q)^{nr}-1)\nmid |t(i_1+(RN^q)^ri_2+\cdots+(RN^q)^{(n-1)r}i_n)|.$$
By Theorem \ref{T4.3}, $t\Lambda_{\omega}$ is a spectrum of $\mu_{M,D}.$
\end{proof}

To this end, we show the sufficiency of Theorem \ref{T1.3}.

\begin{proof}[Proof of the sufficiency of Theorem \ref{T1.3}]
  By Theorem \ref{T4.4}, for any two distinct integers $t_1,t_2$ with gcd$(t_i,N)=1$ for $i=1,2$, there exist $\omega_1,\omega_2$ such that $t_1\Lambda_{\omega_1}$ and $t_2\Lambda_{\omega_2}$ are spectra of $\mu_{M,D}$. Moreover, we can require that $\omega_1=\omega_2$. In fact, similarly as in the proof of Theorem \ref{T4.4}, there exists a common periodic word $\omega$ with sufficiently large periodicity $r$. One has
  $$|(RN^q)^{nr}-1|>|t_i(i_1+(RN^q)^ri_2+\cdots+(RN^q)^{(n-1)r}i_n)|$$
  for $i=1,2$. Then
  $$((RN^q)^{nr}-1)\nmid|t_i(i_1+(RN^q)^ri_2+\cdots+(RN^q)^{(n-1)r}i_n)|$$
  for $i=1,2$. This proves the above claim. Write $t=\frac{t_1}{t_2}$. Let $\Lambda'=t_2\Lambda_{\omega}$. We have $t\Lambda'=t_1\Lambda_{\omega}$. Therefore $\Lambda'$ and $t\Lambda'$ are spectra of $\mu_{M,D}$.
\end{proof}

\section{Examples}\label{5}
As an application of Theorem \ref{T1.2}, we give several examples in this section. Firstly, we have following theorem.

\begin{thm}
With the hypotheses of Theorem \ref{T1.1}, let $\Lambda$ be defined in \eqref{1.3}. If $t$ is a divisor of $R$, then $t^k\Lambda$ is a spectrum of $\mu_{M,D}$ for all non-negative integer $k$.
\end{thm}
\begin{proof}
It follows from $t$ is a divisor of $R$ and $gcd(R,N)=1$ that $gcd(t^k,N)=1$. Since $gcd((RN^q)^n-1,t)=1$ for all $n\in\mathbb{N}$, we have \eqref{1.4} is equivalent to
\begin{eqnarray}\label{5.1}
((RN^q)^n-1)|(i_1+(RN^q)i_2+\cdots+(RN^q)^{n-1}i_n)
\end{eqnarray}
where all $i_j\in L$.
For any  $i_j\in L$, we have
$$
i_1+(RN^q)i_2+\cdots+(RN^q)^{n-1}i_n\leq (RN^q-R)(1+RN^q+\cdots+(RN^q)^{n-1})<(RN^q)^n-1,
$$
which implies there is no $I=i_1i_2\dots i_n(I\neq 0^n)$ with all $i_j\in L$ such that \eqref{5.1} holds for all $n\in \mathbb{N}$. Hence, $t^k\Lambda$ is a spectrum of $\mu_{M,D}$ by Theorem \ref{T1.2}.
\end{proof}

\bigskip

\begin{ex}
Let $R=15$, $q=3$, $N=2$, then the self-similar measures $\mu_{M,D}$ generated by $M=RN^q=120$ and the product-form digit set
$D=\{0,1\}\oplus 2\{0,1\}\oplus 2^2\{0,1\}=\{0,1,2,3,4,5,6,7\}$ is a spectral measure with a spectrum
\begin{eqnarray*}
\Lambda=\left\{\sum_{j=1}^{k}120^{j-1}l_j:k\geq1 \text{ and all }  l_j\in \{0,15,30,45,60,75,90,105\} \right\}.
\end{eqnarray*}
Furthermore, for $t=3^k, 5^k, 15^k$ with $k\in\mathbb{N}$, $t\Lambda$ are also spectra of $\mu_{M,D}$.
\end{ex}
\begin{proof}
Suppose that $t=3^k (k\in\mathbb{N})$, it is easy to see that $gcd(t,N)=1$. By a simple calculation, we have $L=\{0,15,30,45,60,75,90,105\}$. Since $gcd(120^n-1,t)=1$ for all $n\in\mathbb{N}$, we have \eqref{1.4} is equivalent to
\begin{eqnarray}\label{5.2}
(120^n-1)|(i_1+120 i_2+ \cdots+120^{n-1}i_n),
\end{eqnarray}
where all $i_j\in L$.
For any  $i_j\in L$, we have
$$
i_1+120 i_2+ \cdots+120^{n-1}i_n\leq \frac{105}{119}(120^n-1)<120^n-1,
$$
which implies there is no $I=i_1i_2\dots i_n(I\neq 0^n)$ with all $i_j\in L$ such that \eqref{5.2} holds for all $n\in \mathbb{N}$. Hence, $t\Lambda$ is a spectrum of $\mu_{M,D}$ by Theorem \ref{T1.2}. Likewise, for the other two cases $t=5^k$ or $t=15^k$, $t\Lambda$ is also a spectrum of $\mu_{M,D}$.
\end{proof}

\begin{ex}
Let $R=3, N=2, q=4$, then the self-similar measures $\mu_{M,D}$ generated by $M=RN^q=48$ and $D=\{0,1\}\oplus 2^3 \{0,1\}=\{0,1,8,9\}$ is a spectral measure with a spectrum
\begin{eqnarray*}
\Lambda=\left\{\sum_{j=1}^{k}48^{j-1}l_j:k\geq1 \text{ and all }  l_j\in \{0,3,24,27\} \right\}.
\end{eqnarray*}
Besides, $3^k \Lambda (k\in\mathbb{N})$ are also spectra of $\mu_{M,D}$.
\end{ex}
\begin{proof}
Let $t=3^k$ for some non-negative integer $k$. It is easy to see that $gcd(t,N)=1$ and $L=\{0,3,24,27\}$. Since $gcd(48^n-1,t)=1$ for all $n\in\mathbb{N}$, we have \eqref{1.4} is equivalent to
\begin{eqnarray}\label{5.3}
(48^n-1)|(i_1+48 i_2+ \cdots+48^{n-1}i_n),
\end{eqnarray}
where all $i_j\in L$.
For any  $i_j\in L$, we have
$$
i_1+48 i_2+ \cdots+48^{n-1}i_n\leq \frac{27}{37}(48^n-1)<48^n-1,
$$
which implies there is no $I=i_1i_2\dots i_n(I\neq 0^n)$ with all $i_j\in L$ such that \eqref{5.3} holds for all $n\in \mathbb{N}$. Hence, $t\Lambda$ is a spectrum of $\mu_{M,D}$ by Theorem \ref{T1.2}.
\end{proof}

\begin{ex}
Let $R=3$, $N=q=2$, then the self-similar measures $\mu_{M,D}$ generated by $M=RN^q=12$ and $D=\{0,1\}\oplus 2 \{0,1\}=\{0,1,2,3\}$ is a spectral measure with a spectrum
\begin{eqnarray*}
\Lambda=\left\{\sum_{j=1}^{k}12^{j-1}l_j:k\geq1 \text{ and all }  l_j\in \{0,3,6,9\} \right\}.
\end{eqnarray*}
If $t$ is an odd number and $11|t$, then $t^k\Lambda$ is not a spectrum of $\mu_{M,D}$ for all non-negative integer $k$. Moreover, except for the number $11$, all positive odd numbers $t$ smaller than $20$ can make $t\Lambda$ a spectrum of $\mu_{M,D}$.
\end{ex}
\begin{proof}
Let $t$ be an odd number and $11|t$, then $(12-1)|(3t^k)$ and $L=\{0,3,6,9\}$.  By Theorem \ref{T1.2}, $t^k\Lambda$ is not a spectrum of $\mu_{M,D}$ for all non-negative integer $k$. It is easy to check $\#\{T(RN^q,tL)\cap\mathbb{Z}\}\leq \lfloor \frac{9t}{11} \rfloor +1$. With some calculations, for all odd numbers $t(t\neq 11,0<t<20)$, there is no $I=i_1i_2{\dots}i_n(I\neq{0^n})$ with all $i_j\in L$ such that \eqref{1.4} holds for all $n\leq\lfloor \frac{9t}{11} \rfloor +1$. Hence, $t\Lambda$ is a spectrum of $\mu_{M,D}$ by Theorem \ref{T1.2}.

\end{proof}


\begin{thebibliography}{9999}

\smallskip
\bibitem{A21} W.H. Ai, Number theory problems related to the spectrum of Cantor-type measures with consecutive digits, Bull. Aust. Math. Soc. 103 (2021), 113-123.
\bibitem{ADH22} L.X. An, X.H. Dong, X.G. He, On spectra and spectral eigenmatrix problems of the planar Sierpinski measures, Indiana Univ. Math. J. 71 (2022), 913-952.
\bibitem{AFL19} L.X. An, X.Y. Fu, C.K. Lai, On spectral Cantor-Moran measures and a variant of Bourgain's sum of sine problem, Adv. Math. 349 (2019), 84-124.
\bibitem{AH14} L.X. An, X.G. He, A class of spectral Moran measures, J. Funct. Anal. 266 (2014), 343-354.
\bibitem{CL23} M.L. Chen, J.C. Liu, On spectra and spectral eigenmatrices of self-affine measures on $\mathbb R^{n}$. Bull. Malays. Math. Sci. Soc. 46 (2023), no. 5, Paper No. 162, 18 pp.
\bibitem{CLL23} Z.C. Chi, Q. Li, J. Lv, Number theoretic considerations related to the scaling spectrum of self-similar measures with consecutive digits, Chaos Solitons Fractals 173 (2023), Paper No. 113648, 10 pp.
\bibitem{D12} X.R. Dai, When does a Bernoulli convolution admit a spectrum? Adv. Math. 231 (2012), 1681-1693.
\bibitem{Dai16} X.R. Dai, Spectra of Cantor measures, Math. Ann. 366 (2016), 1621-1647.
\bibitem{DHL13} X.R. Dai, X.G. He, C.K. Lai, Spectral property of Cantor measures with consecutive digits, Adv. Math. 242 (2013), 187-208.
\bibitem{DHL14} X.R. Dai, X.G. He, K.S. Lau, On spectral $N$-Bernoulli measures, Adv. Math. 259 (2014), 511-531.
\bibitem{DL15} Q.R. Deng, K.S. Lau, Sierpinski-type spectral self-similar measures, J. Funct. Anal. 269 (2015), 1310-1326.
\bibitem{DH16} D. Dutkay, J. Haussermann, Number theory problems from the harmonic analysis of a fractal, J. Number Theory 159 (2016), 7-26.
\bibitem{DHL19} D. Dutkay, J. Haussermann, C.K. Lai, Hadamard triples generate self-affine spectral measures, Trans. Amer. Math. Soc. 371 (2019), 1439-1481.
\bibitem{DHS09} D. Dutkay, D.G. Han, Q.Y. Sun, On the spectra of a Cantor measure, Adv. Math. 221 (2009), 251-276.
\bibitem{DHS14} D. Dutkay, D.G. Han, Q.Y. Sun, Divergence of the mock and scrambled Fourier series on fractal measures, Trans. Amer. Math. Soc. 366 (2014), 2191-2208.
\bibitem{DJ12} D. Dutkay, P.E.T. Jorgensen, Fourier duality for fractal measures with affine scales, Math. Comp. 81 (2012), 2253-2273.
\bibitem{F74} B.Fuglede, Commuting self-adjoint partial differential operators and a group theoretic problem, J. Funct. Anal. 16 (1974), 101-121.
\bibitem{FH17} Y.S. Fu, L. He, Scaling of spectra of a class of random convolution on $\mathbb{R}$, J. Funct. Anal. 273 (2017), 3002-3026.
\bibitem{FHW18} Y.S. Fu, X.G. He, Z.X. Wen, Spectra of Bernoulli convolutions and random convolutions, J. Math. Pures Appl. 116 (2018), 105-131.
\bibitem{FTW22} Y.S. Fu, M.W. Tang, Z.Y. Wen, Convergence of mock Fourier series on generalized Bernoulli convolutions, Acta Appl. Math. 179 (2022), Paper No. 14.
\bibitem{H81} J. Hutchinson, Fractals and self-similarity, Indiana Univ. Math. J. 30 (1981), 713-747.
\bibitem{HTW19} X.G. He, M.W. Tang, Z.Y. Wu, Spectral structure and spectral eigenvalue problems of a class of self-similar spectral measures, J. Funct. Anal. 277 (2019), no. 10, 3688-3722.
\bibitem{JKS11} P.E.T. Jorgensen, K. Kornelson, K. Shuman, Families of spectral sets for Bernoulli convolutions, J. Fourier Anal. Appl. 17 (2011), 431-456.
\bibitem{JKS14} P.E.T. Jorgensen, K. Kornelson, K. Shuman, Scaling by 5 on a $\frac{1}{4}$-Cantor measure, Rocky Mountain J. Math. 44 (2014), 1881-1901.
\bibitem{JP98} P.E.T. Jorgensen, S. Pedersen, Dense analytic subspaces in fractal $L^2$ spaces, J. Anal. Math. 75 (1998), 185-228.
\bibitem{KM2006} M.N. Kolountzakis, M. Matolcsi, Complex Hadamard matrices and the spectral set conjecture, Collect. Math. Extra (2006), 281-291.
\bibitem{KM06} M.N. Kolountzakis, M. Matolcsi, Tiles with no spectra, Forum Math. 18 (2006), 519-528.
\bibitem{LW02} I. {\L}aba, Y. Wang, On spectral Cantor measures, J. Funct. Anal. 193 (2002), 409-420.
\bibitem{LL21} H.X. Li, Q. Li, Multiple spectra of self-similar measures with three digits on $\mathbb R$, Acta Math. Hungar. 164 (2021), no.1, 296-311.
\bibitem{LX17} J.L. Li, D. Xing, Multiple spectra of Bernoulli convolutions. Proc. Edinb. Math. Soc. (2), 60 (2017), 187-202.
\bibitem{LW22} J.J. Li, Z.Y. Wu, On spectral structure and spectral eigenvalue problems for a class of self similar spectral measure with product form, Nonlinearity 35 (2022), 3095-3117.
\bibitem{LA23} S.J. Li, W.H. Ai, Spectral eigenmatrix of the planar spectral measures with four elements, Anal. Math. 49 (2023), 545-562.
\bibitem{LMW22} W.X. Li, J.J. Miao, Z.Q. Wang, Weak convergence and spectrality of infinite convolutions. Adv. Math. 404 (2022), part B, Paper No. 108425, 26 pp.
\bibitem{LTW23} J.C. Liu, M.W. Tang, S. Wu, The spectral eigenmatrix problems of planar self-affine measures with four digits. Proc. Edinb. Math. Soc. (2) 66 (2023), no. 3, 897-918.
\bibitem{LWZ23} J.F. Lu, S. Wang, M.M. Zhang, Self-similar measures with product-form digit sets and their spectra, J. Math. Anal. Appl. 527 (2023), no. 1, Paper No. 127340, 15 pp.
\bibitem{LLW22} J. Lv, Q. Li, S.D. Wei, Number theoretic related to the scaling spectrum of self-similar measure with three element digit sets. Complex Anal. Oper. Theory 16 (2022), no. 7, Paper No. 97, 23 pp.
\bibitem{PA23} W.Y. Pan, W.H. Ai, Divergence of mock Fourier series for spectral measures. Proc. Roy. Soc. Edinburgh Sect. A 153 (2023), no. 6, 1818-1832.
\bibitem{S00} R.S. Strichartz, Mock Fourier series and transforms associated with certain Cantor measures, J. Anal. Math. 81 (2000), 209-238.
\bibitem{S06} R.S. Strichartz, Convergence of mock Fourier series, J. Anal. Math. 99 (2006), 333-353.
\bibitem{T04} T. Tao, Fuglede's conjecture is false in 5 and higher dimensions, Math. Res. Lett. 11 (2004), 251-258.
\bibitem{WDA19} Z.M. Wang, X.H.Dong, W.H.Ai, Scaling of spectra of a class of self-similar measures on $\mathbb{R}$, Math. Nachr. 292 (2019), 2300-2307.
\bibitem{WW20} C. Wang, Z.Y. Wu, On spectral eigenvalue problem of a class of self-similar spectral measures with consecutive digits, J. Fourier Anal. Appl. 26 (2020), Paper No. 82, 18 pp.
\bibitem{W19} Z.Y. Wu, On spectral eigenvalue problem of a class of generalized Cantor measures, J. Math. Anal. Appl. 480 (2019), 123374, 11 pp.
\bibitem{WZ18} Z.Y. Wu, M. Zhu, Scaling of spectra of self-similar measures with consecutive digits, J. Math. Anal. Appl. 459 (2018), 307-319.
\bibitem{ZAC23} S.N. Zeng, W.H. Ai, J.L. Chen, The spectra of Cantor-type measures with consecutive digits, Bull. Malays. Math. Sci. Soc. 46 (2023), Paper No. 125, 16 pp.
\end{thebibliography}
\end{document}